\UseRawInputEncoding
\documentclass{amsart}
\usepackage{setspace, amsmath, amsthm, amssymb, amsfonts, amscd, epic, graphicx, ulem, dsfont}
\usepackage[T1]{fontenc}
\usepackage{multirow}
\usepackage{bbm}
\usepackage{enumerate}

\makeatletter \@namedef{subjclassname@2010}{
  \textup{2020} Mathematics Subject Classification}
\makeatother

\newtheorem{thm}{Theorem}[section]

\newtheorem{cor}[thm]{Corollary}
\newtheorem{lem}[thm]{Lemma}
\newtheorem{pro}[thm]{Proposition}

\theoremstyle{remark}
\newtheorem*{rema}{Remark}

\theoremstyle{definition}

\newtheorem{exa}[thm]{\textbf{Example}}

\newcommand{\ran}{\operatorname{ran}}

\newcommand{\R}{\mathbb{R}}

\newcommand{\N}{\mathbb{N}}

\newcommand{\C}{\mathbb{C}}

\begin{document}

\title[Operators with self-adjoint or normal powers]{Unbounded operators having self-adjoint or normal powers and some related results}
\author[S. Dehimi and M. H. Mortad]{Souheyb Dehimi and Mohammed Hichem Mortad$^*$}

\dedicatory{}
\thanks{* Corresponding author.}
\date{}


\address{(The first author) Department of Mathematics, Faculty of Mathematics and Informatics,
University of Mohamed El Bachir El Ibrahimi, Bordj Bou Arréridj,
El-Anasser 34030, Algeria.}

\email{souheyb.dehimi@univ-bba.dz, sohayb20091@gmail.com}

\address{(The corresponding and second author) Laboratoire d'analyse mathématique et applications. Département
de Mathématiques, Université Oran 1, Ahmed Ben Bella, B.P.
1524, El Menouar, Oran 31000, Algeria.}

\email{mhmortad@gmail.com, mortad.hichem@univ-oran1.dz.}

\begin{abstract}
We show that a densely defined closable operator $A$ such that the
resolvent set of $A^2$ is not empty, is necessarily closed. This
result is then extended to the case of a polynomial $p(A)$. We also
generalize a recent result by Sebestyén-Tarcsay concerning the
converse of a result by J. von Neumann. Other interesting
consequences are also given, one of them being a proof that if $T$
is a quasinormal (unbounded) operator such that $T^n$ is normal for
some $n\geq2$, then $T$ is normal. By a recent result by
Pietrzycki-Stochel, we infer that a closed subnormal operator such
that $T^n$ is normal, must be normal.

Another remarkable result is the fact that a hyponormal operator
$A$, bounded or not, such that $A^p$ and $A^q$ are self-adjoint for
some co-prime numbers $p$ and $q$, is self-adjoint. It is also shown
that an invertible operator (bounded or not) $A$ for which $A^p$ and
$A^q$ are normal for some co-prime numbers $p$ and $q$, is normal.
These two results are shown using Bézout's theorem in
arithmetic.
\end{abstract}

\maketitle

\section*{Notation}

First, we assume that readers have some familiarity with the
standard notions and results in operator theory (see e.g.
\cite{Mortad-Oper-TH-BOOK-WSPC} and \cite{SCHMUDG-book-2012} for
some background). We do recall most of the needed notions though.
First, note that in this paper all operators are linear.

Let $H$ be a complex Hilbert space and let $B(H)$ be the algebra of
all bounded linear operators defined from $H$ into $H$.

If $S$ and $T$ are two linear operators with domains $D(S)\subset H$
and $D(T)\subset H$ respectively, then $T$ is said to be an
extension of $S$, written $S\subset T$, when $D(S)\subset D(T)$ and
$S$ and $T$ coincide on $D(S)$.

The product $ST$ and the sum $S+T$ of two operators $S$ and $T$ are
defined in the usual fashion on the natural domains:

\[D(ST)=\{x\in D(T):~Tx\in D(S)\}\]
and
\[D(S+T)=D(S)\cap D(T).\]

When $\overline{D(T)}=H$, we say that $T$ is densely defined. In
such case, the adjoint $T^*$ exists and is unique.

An operator $T$ is called closed if its graph is closed in $H\oplus
H$. $T$ is called closable if it has a closed extension. If $T$ is
densely defined, then $T$ is closable iff $T^*$ is densely defined.
The smallest closed extension of $T$ is called its closure, noted
$\overline{T}$.

We say that $T$ is symmetric if
\[<Tx,y>=<x,Ty>,~\forall x,y\in D(T).\]
If $T$ is densely defined, we say that $T$ is self-adjoint when
$T=T^*$; normal if $T$ is \textit{closed} and $TT^*=T^*T$. Observe
that a densely defined operator $T$ is symmetric iff $T\subset T^*$.

A symmetric operator $T$ is called positive if
\[<Tx,x>\geq 0, \forall x\in D(T).\]
Recall that the absolute value of a closed $T$ is given by
$|T|=\sqrt{T^*T}$ where $\sqrt{\cdot}$ designates the unique
positive square root of $T^*T$, which is self-adjoint and positive
by the closedness of $T$. It is also known that $D(T)=D(|T|)$ when
$T$ is closed and densely defined.

We say that $B\in B(H)$ commutes with a linear operator $A$ with
domain $D(A)\subset H$ when $BA\subset AB$, that is when $BAx=ABx$
for all $x\in D(A)\subset D(AB)$.

Let $A$ be an injective operator (not necessarily bounded) from
$D(A)$ into $H$. Then $A^{-1}: \ran(A)\rightarrow H$ is called the
inverse of $A$, with $D(A^{-1})=\ran(A)$.

If the inverse of an unbounded operator is bounded and everywhere
defined (e.g. if $A:D(A)\to H$ is closed and bijective), then $A$ is
said to be boundedly invertible. In other words, such is the case if
there is a $B\in B(H)$ such that
\[AB=I\text{ and } BA\subset I.\]
If $A$ is boundedly invertible, then it is closed. Recall also that
$T+S$ is closed (resp. closable) if $S\in B(H)$ and $T$ is closed
(resp. closable), and that $ST$ is closed (resp. closable) if
 e.g. $S$ is closed (resp. closable) and $T\in B(H)$.

Based on the bounded case and the previous definition, we say that
an unbounded $A$ with domain $D(A)\subset H$ is right invertible if
there exists an everywhere defined $B\in B(H)$ such that $AB=I$; and
we say that $A$ is left invertible if there is an everywhere defined
$C\in B(H)$ such that $CA\subset I$. Clearly, if $A$ is left and
right invertible simultaneously, then $A$ is boundedly invertible.

The spectrum of unbounded operators is defined as follows: Let $A$
be an operator. The resolvent set of $A$, denoted by $\rho(A)$, is
defined by
\[\rho(A)=\{\lambda\in\C:~\lambda I-A\text{ is bijective and }(\lambda I-A)^{-1}\in B(H)\}.\]
The complement of $\rho(A)$, denoted by $\sigma(A)$, i.e.
\[\sigma(A)=\C\setminus \rho(A)\]
is called the spectrum of $A$.

Clearly, $\lambda\in \rho(A)$ iff there is a $B\in B(H)$ such that
\[(\lambda I-A)B=I\text{ and } B(\lambda I-A)\subset I.\]

Also, recall that if $A$ is a linear operator which is not closed,
then $\sigma(A)=\C$. Recall also that
\[\sigma(A^n)=[\sigma(A)]^n\]
when $A$ is closed ($n\in\N$).

Let us now recall some rudimentary facts about matrices of
non-necessarily bounded operators. Let $H$ and $K$ be two Hilbert
spaces and let $A:H\oplus K\to H\oplus K$ (we may also use $H\times
K$ instead of $H\oplus K$) be defined by
\begin{equation}\label{matrix reprenstation UNBD EQU}
A=\left(
      \begin{array}{cc}
        A_{11} & A_{12} \\
        A_{21} & A_{22} \\
      \end{array}
    \right)
\end{equation}
where $A_{11}\in L(H)$, $A_{12}\in L(K,H)$, $A_{21}\in L(H,K)$ and
$A_{22}\in L(K)$ are not necessarily bounded operators. If $A_{ij}$
has a domain $D(A_{ij})$ with $i,j=1,2$, then
\[D(A)=(D(A_{11})\cap D(A_{21}))\times (D(A_{12})\cap D(A_{22}))\]
is the natural domain of $A$. So if $(x_1,x_2)\in D(A)$, then
\[A\left(
     \begin{array}{c}
       x_1 \\
       x_2 \\
     \end{array}
   \right)=\left(
             \begin{array}{c}
               A_{11}x_1+A_{12}x_2\\
               A_{21}x_1+A_{22}x_2 \\
             \end{array}
           \right).
\]
As is customary, we allow the abuse of notation $A(x_1,x_2)$.
Readers should also be careful when dealing with products of
matrices of (unbounded) operators as they may encounter some issues
with their domains.

Also, recall that the adjoint of $\left(
      \begin{array}{cc}
        A_{11} & A_{12} \\
        A_{21} & A_{22} \\
      \end{array}
    \right)$ is not always $\left(
      \begin{array}{cc}
        A^*_{11} & A^*_{21} \\
        A^*_{12} & A^*_{22} \\
      \end{array}
    \right)$ (even when all domains are dense including the main domain $D(A)$) as
    known counterexamples show. Nonetheless, e.g.
    \[\left(
        \begin{array}{cc}
          A & 0 \\
          0 & B \\
        \end{array}
      \right)^*=\left(
        \begin{array}{cc}
          A^* & 0 \\
          0 & B^* \\
        \end{array}
      \right)\text{ and } \left(
                            \begin{array}{cc}
                              0 & C \\
                              D & 0 \\
                            \end{array}
                          \right)^*=\left(
                            \begin{array}{cc}
                              0 & D^* \\
                              C^* & 0 \\
                            \end{array}
                          \right)
    \]
if $A$, $B$, $C$ and $D$ are all densely defined. See e.g.
\cite{tretetr-book-BLOCK} for proofs of other results which will be
used below, and for more about unbounded operator matrices.

In the end, we recall some definitions of unbounded non-normal
operators. A densely defined operator $A$ with domain $D(A)$ is
called hyponormal if
\[D(A)\subset D(A^*)\text{ and } \|A^*x\|\leq\|Ax\|,~\forall x\in D(A).\]

A densely defined linear operator $A$ with domain $D(A)\subset H$,
is said to be subnormal when there are a Hilbert space $K$ with
$H\subset K$, and a normal operator $N$ with $D(N)\subset K$ such
that
\[D(A)\subset D(N)\text{ and } Ax=Nx \text{ for all} x\in D(A).\]

A quasinormal operator $A$ is a closed densely defined one such that
$AA^*A=A^*A^2$ (or $AA^*A\subset A^*A^2$ as in say \cite{Jablonski
et al 2014}).

A linear operator $A:D(A)\subset H\to H$ is said to be paranormal if
\[\|Ax\|^2\leq \left\|A^2x\right\|\|x\|\]
for all $x\in D(A^2)$.

As in the bounded case, we have (see e.g. \cite{Jablonski et al
2014}, \cite{Janas-HYPO-I}, \cite{Madjak} and
\cite{Stochel-Sza-subnormality-BOOK-unpublished?}):
\[\text{Normal}\subsetneq \text{Quasinormal}\subsetneq\text{Subnormal}\subsetneq \text{Hyponormal}\subsetneq \text{Paranormal}.\]

Observe that a hyponormal operator is necessarily closable, hence so
are quasinormal and subnormal operators. Paranormal operators,
however, need not be closable. See
\cite{Daniluk-paranormals-non-closable} for a counterexample. In
fact, in \cite{Mortad-paranormal-daniluk}, the writer found a
densely defined paranormal operator $A$ such that downright
$D(A^*)=\{0\}$.

\section{Introduction}

As is known, if $A$ is a closed operator, then $A^2$ need not be
closed. There are known counterexamples, see e.g.
\cite{Mortad-34page paper square roots et al.}. On the other hand,
there are closable operators having closed squares, e.g. if
$\mathcal{F}_0$ denotes the restriction of the $L^2(\R)$-Fourier
transform to the dense subspace $C_0^{\infty}(\R)$ (the space of
infinitely differentiable functions with compact support), then it
is well known that
\[D(\mathcal{F}_0^2)=\{0\}.\]
Hence $\mathcal{F}_0$ is unclosed whilst $\mathcal{F}_0^2$ is
trivially closed on $\{0\}$. In fact, there is an unclosable
operator $A$, which is everywhere defined and such that $A^2=0$ on
all of $H$, see again \cite{Mortad-34page paper square roots et
al.}.

If $A$ is a closable operator such that $\overline{A}^2$ is
self-adjoint, then $A^2$ need not be self-adjoint. A simple example
is to take $A$ to be the restriction of the identity operator $I$
(on $H$) to some dense (non-closed) subspace $D$ of $H$. Denote this
restriction by $I_D$. Then $\overline{A}^2=I$ fully on $H$ and so
$\overline{A}^2$ is self-adjoint. However, $A^2$ is not self-adjoint
for $A^2=I_D$ and so $A^2$ is not even closed.

What about the converse? I.e. if $A$ is closable and $A^2$ is
self-adjoint, then could it be true that $\overline{A}^2$ is
self-adjoint? A positive answer can be obtained if one comes to show
that if $A$ is a closable operator with a self-adjoint square, then
$A$ is closed. This question, and other related and more general
ones are investigated in the present paper.

There are known conditions for which $p(A)$ is closed whenever $A$
is closed, where $p$ is a complex polynomial in one variable of
degree $n$ say. For instance, if $A$ is a closed operator in some
Hilbert (or even Banach) space with domain $D(A)$ such that
$\sigma(A)\neq\C$, then $p(A)$ is closed on $D(A^n)$.  See e.g. pp.
347-348 in \cite{Dautray-Lions-VOL2}. Also, K. Schm\"{u}dgen showed
in \cite{SCHMUDG-1983-An-trivial-domain} that if $A$ is a densely
defined closed and symmetric operator, then $p(A)$ is closed. J.
Stochel proved this result in the case of unbounded subnormal
operators in \cite{Stochel-IEOT-2002}. A related paper to some of
our results (especially Theorem \ref{26/07/2020 THM}) is
\cite{Stochel-Sza-domination-2003}. In fact, in an unpublished work
yet \cite{Stochel-Sza-subnormality-BOOK-unpublished?}, it is shown
that if $A$ is a paranormal closed operator, then $A^n$ is closed
for every $n\in\N$. It is one of our aims to investigate the
converse of some of these results.

Before announcing an interesting application of some of our results,
recall first a well-known result by the legendary J. von Neumann
stating that if $T$ is a densely defined closed operator, then both
$TT^*$ and $T^*T$ are self-adjoint (and positive). Amazingly, no one
had studied the converse until very recently, i.e. until
\cite{Sebestyen-Tarcsay-TT* von Neumann T closed}. The outcome is
quite striking. Indeed, in the previous reference, the writers
Sebestyén-Tarcsay discovered that if $TT^*$ and $T^*T$ are both
self-adjoint, then $T$ must be closed. Then Gesztesy-Schm\"{u}dgen
 provided in \cite{Gesztesy-Schmudgen-AA*} a simpler proof based on a
technique using matrices of unbounded operators. Notice that
Gesztesy-Schm\"{u}dgen's proof only work for complex Hilbert spaces
while the original proof by Sebestyén-Tarcsay works also for real
Hilbert spaces just as good. To end this remark, Sebestyén-Tarcsay
also gave a proof of their result using block operator matrices as
well as some other results (see e.g. Theorem 8.1 in
\cite{Sebestyen-Tarcsay-LMA-2019}). This result with the two
different proofs is referred to here as the
Sebestyén-Tarcsay-Gesztesy-Schm\"{u}dgen reversed von Neumann
theorem. Note that the self-adjointness of only one of $TT^*$ and
$T^*T$ is not sufficient to guarantee the closedness of $T$. This
was already noted in \cite{sebestyen-tarcsay-TT* has an extension}.
Herein, we show that if $T$ is a densely defined closable such that
$\sigma(TT^*)\cup \sigma(T^*T)\neq \C$, then $T$ is necessarily
closed.

Other consequences are also given. For example, we show that if $T$
is a quasinormal (unbounded) operator such that $T^n$ is normal for
some $n\geq2$, then $T$ must be normal. By a recent result by
Pietrzycki-Stochel in
\cite{Pietrzycki-Stochel-follow-up-2020-Conjecture curto el al.}, we
deduce that a closed subnormal operator such that $T^n$ is normal
for some $n$, is necessarily normal. These results are closely
related to others which have been of some interest recently (see
\cite{CURTO-et-AL-JFA-2020} and \cite{Pietrzycki-2020-Conjecture
curto el al.}).

Other remarkable results are shown by invoking Bézout's theorem
in arithmetic. For instance, if a hyponormal operator $A$, bounded
or not, is such that $A^p$ and $A^q$ are self-adjoint for some
co-prime numbers $p$ and $q$, then it is self-adjoint. By the same
token, it is  shown that an invertible operator (bounded or not) $A$
for which $A^p$ and $A^q$ are normal for some co-prime numbers $p$
and $q$, is normal. Another application of Bézout's theorem is:
If a boundedly invertible $T$ commutes with both $A^p$ and $A^q$,
for some relatively prime numbers $p$ and $q$, then $T$ commutes
with $A$.

\section{Main Results: The case of monomials}

We choose to deal first with the case of squares for it is closely
related to the important notion of square roots. Then we give the
generalizations to $p(A)$ as long as these generalizations are known
to hold for us.

\begin{thm}\label{THM main A2 closed}
Let $A$ be a closable densely defined operator with domain
$D(A)\subset H$ (where $H$ could also be a Banach space here) such
that $\sigma(A^2)\neq \C$. Then $A$ is closed.
\end{thm}

\begin{proof}
Let $\lambda\in \C\setminus \sigma(A^2)$. Then $A^2-\lambda I$ is
boundedly invertible. Let $\alpha$ be complex and such that
$\alpha^2=\lambda$ and write
\[A^2-\lambda I=(A-\alpha I)(A+\alpha I)=(A+\alpha I)(A-\alpha I).\]
Letting $B$ to be the bounded inverse of $A^2-\lambda I$, we see
that
\[I=(A^2-\lambda I)B=(A-\alpha I)(A+\alpha I)B.\]
Since $A$ is closable, so is $A+\alpha I$. Hence $(A+\alpha I)B$ too
is closable. But
\[H=D[(A-\alpha I)(A+\alpha I)B]\subset D[(A+\alpha I)B],\]
whereby $(A+\alpha I)B$ becomes everywhere defined. Therefore,
$(A+\alpha I)B$ must be in $B(H)$ by invoking the closed graph
theorem. In other words, $A-\alpha I$ is right invertible, i.e. it
possesses an everywhere defined bounded right inverse. Now, by the
first displayed formula, we see that $A-\alpha I$ is one-to-one as
well. Thus, $A-\alpha I$ is bijective. So, let $C$ be a left inverse
(a priori not necessarily bounded) of $A-\alpha I$. Clearly, $C$
must be defined on all of $H$. Moreover,
\[C(A-\alpha I)\subset I\Longrightarrow C(A-\alpha I)(A+\alpha I)B\subset (A+\alpha I)B\Longrightarrow C\subset (A+\alpha I)B\]
and so $C=(A+\alpha I)B$ as they are both defined on all of $H$.
This actually says that $A-\alpha I$ is boundedly invertible.
Accordingly, $A-\alpha I$ is closed or merely $A$ is closed, as
needed.
\end{proof}

The following simple consequence seems to have some interest.

\begin{cor}
If $A$ is a closable (unclosed) operator such that $A^2$ is closed,
then
\[\sigma(A^2)=\C.\]
\end{cor}

\begin{proof}If $\sigma(A^2)\neq \C$, Theorem \ref{THM main A2
closed} yields the closedness of $A$ which is a contradiction.

\end{proof}

\begin{rema}
The assumption $\sigma(A^2)\neq \C$ made in Theorem \ref{THM main A2
closed} may not just be dropped. A simple counterexample is to
consider a closable non-closed $A$ such that $A^2$ is closable but
unclosed. Then $\sigma(A^2)=\C$. An explicit example would be to
take $A=I_D$, the identity operator restricted to some dense
subspace $D$.
\end{rema}

\begin{rema}The closability is indispensable for the result to hold.
For example and as alluded to above, there are non-closable
\textit{everywhere defined} operators $T$ (i.e. $D(T)=H$) such that
$T^2=0$ everywhere on $H$. Clearly
\[\sigma(T^2)=\{0\}\neq\C\]
and yet $T$ is not even closable.

Similarly, there are everywhere defined unclosable operators $T$
such that $T^2=I$ on all of $H$, and hence $\sigma(T^2)=\{1\}$.
These examples may be consulted in \cite{Mortad-34page paper square
roots et al.}.

\end{rema}

\begin{cor}\label{A2 s.a. A closable A closed}
Let $A$ be a closable densely defined operator such that $A^2$ is
self-adjoint. Then $A$ is closed.
\end{cor}

\begin{proof}Recall that a self-adjoint operator has always a real spectrum (see \cite{Boucif-Dehimi-Mortad} for a new proof). Since $A^2$ is self-adjoint, $\sigma(A^2)\subset \R$.
Now, apply Theorem \ref{THM main A2 closed}.
\end{proof}

It was shown in \cite{Sebestyen-Tarcsay-self-adjoint squares} (and
also in \cite{Sebestyen-Tarcsay-LMA-2019}) that a symmetric operator
having a self-adjoint (and positive!) square must be self-adjoint as
well. As a consequence of Corollary \ref{A2 s.a. A closable A
closed}, we have a different proof of this result.

\begin{pro}\label{25/0/2020}
Let $A$ be a symmetric operator (not necessarily densely defined)
such that $A^2$ is self-adjoint. Then $A$ too is self-adjoint.
\end{pro}

\begin{proof}First, as $A^2$ is self-adjoint, it is densely defined, and hence so is $A$. So $A$ is a densely defined symmetric operator, i.e. it becomes closable. Let us show that $\overline{A}$ is self-adjoint.
Indeed, since $A\subset A^*$, we have
\[\overline{A}\subset A^*=\overline{A}^*.\]
Hence
\[A^2\subset \overline{A}^2\subset \overline{A}^*\overline{A}.\]
By the self-adjointness of both $A^2$ and
$\overline{A}^*\overline{A}$ as well as a known maximality argument,
it ensues that
\[\overline{A}^2=\overline{A}^*\overline{A}.\]

Theorem 3.2 in \cite{Dehimi-Mortad-Tarcsay-1} now intervenes and
gives the self-adjointness of $\overline{A}$. But Corollary \ref{A2
s.a. A closable A closed} gives the closedness of $A$, that is, $A$
is self-adjoint.
\end{proof}

In other language, we have re-shown that a symmetric square root of
a self-adjoint (necessarily positive!) is itself self-adjoint. Next,
we generalize this result to normal operators.

\begin{cor}
Let $A$ be a symmetric operator such that $A^2$ is normal. Then $A$
is self-adjoint.
\end{cor}

\begin{proof}That $A$ is densely defined is clear. Since $A\subset
A^*$, we have
\[A^2\subset (A^*)^2\subset (A^2)^*.\]

In other words, $A^2$ is symmetric. But, and as is known, symmetric
operators that are normal are self-adjoint. Therefore, $A^2$ is
self-adjoint, and so $A$ too is self-adjoint by the foregoing
results.
\end{proof}

\begin{rema}
The above proof is algebraic. There is also a spectral proof.
Indeed, since $A$ is symmetric, it follows that $A^2$ is positive.
However, a normal operator with a positive spectrum is self-adjoint.
Then Proposition \ref{25/0/2020} does the remaining job.
\end{rema}

The following result contains some important conclusions (observe
that it generalizes Proposition \ref{25/0/2020}).

\begin{thm}\label{THM BIG THM ANOTHER}
Let $A$ be a linear operator with domain $D(A)$ such that $A^2$ is
self-adjoint and $D(A)\subset D(A^*)$. Then $A$ is closed,
$D(A)=D(A^*)$, $(A^*)^2=A^2$ and $D(AA^*)=D(A^*A)$.
\end{thm}

\begin{proof}
As above, $A$ is densely defined. Since $D(A)\subset D(A^*)$, it
follows that $A$ is closable. Hence $A$ is closed by Corollary
\ref{A2 s.a. A closable A closed}. So, it only remains to show that
$D(A)=D(A^*)$. To this end, observe that
\[(A^*)^2\subset (A^2)^*=A^2.\]
Since $D(A)\subset D(A^*)$, we get
\begin{equation}\label{mlmlmlkjjkhhjdfkjfgdjkfghljgfhoigotook,n,c,,c,c,}
D(AA^*)\subset D[(A^*)^2]\subset D(A^2)\subset D(A^*A).
\end{equation}

By Theorem 9.4 in \cite{Weidmann}, we obtain $D(\sqrt{AA^*})\subset
D(\sqrt{A^*A})$ for $AA^*$ and $A^*A$ are self-adjoint and positive.
But $\sqrt{AA^*}=|A^*|$ and $\sqrt{A^*A}=|A|$. So, by the closedness
of both $A$ and $A^*$, we finally infer that
\[D(A^*)=D(|A^*|)\subset D(|A|)=D(A)\]
thereby $D(A)=D(A^*)$.

To prove the other claimed properties, start with the inclusion
$(A^*)^2\subset A^2$ which was obtained above. Since $A^2$ is
self-adjoint, $\varnothing\neq \sigma(A^2)\subset\R$. So, let
$\lambda\in\sigma(A^2)$ and so $\lambda=\mu^2$ for some
$\mu\in\sigma(A)$. Hence $\overline{\mu}\in\sigma(A^*)$. Therefore,
\[\lambda=\overline{\lambda}=\overline{\mu}^2\in [\sigma(A^{*})]^2=\sigma({A^{*}}^2).\]
That is, $\rho ({A^{*}}^2)\subset \rho(A^2)$. On the other hand,
$\rho ({A^{*}}^2)\neq \varnothing$ as it is easy to see that
$\sigma({A^{*}}^2)\neq\C$. So let $\alpha\in \rho ({A^{*}}^2)$ and
write
\[(A^*)^2-\alpha I\subset A^2-\alpha I.\]
Since $(A^*)^2-\alpha I$ is onto and $A^2-\alpha I$ is one-to-one,
by a simple maximality result (see Lemma 1.3 in
\cite{SCHMUDG-book-2012}), it follows that
\[(A^*)^2-\alpha I=A^2-\alpha I\]
or merely $(A^*)^2=A^2$. Finally, as $D(A)=D(A^*)$ and
$(A^*)^2=A^2$, Inclusions
(\ref{mlmlmlkjjkhhjdfkjfgdjkfghljgfhoigotook,n,c,,c,c,}) become
\[D(AA^*)=D[(A^*)^2]=D(A^2)=D(A^*A),\]
marking the end of the proof.
\end{proof}

\begin{cor}\label{17/05/2021 CORO}
Let $A$ be an unbounded hyponormal operator such that $A^2$ is
self-adjoint and positive. Then $A$ too is self-adjoint.
\end{cor}

\begin{proof}Since $A$ is closable and $A^2$ is self-adjoint, it follows that $A$
is closed. Recall also that closed hyponormal operators having a
real spectrum are self-adjoint (see the proof of Theorem 8 in
\cite{Dehimi-Mortad-BKMS}). Let $\lambda\in\sigma(A)$ and so
$\lambda^2\in\sigma(A^2)$, i.e. $\lambda^2\geq0$ because $A^2$ is
positive. Thus $\lambda$ must be real. Consequently, $A$ is
self-adjoint.
\end{proof}

\begin{pro}\label{wxxxxxccccccccccghshgsjhgjshhdkqmfùùgh*}
Let $A$ be an unbounded quasinormal operator with $A^2={A^*}^2$.
Then $A$ is normal.
\end{pro}

\begin{proof}We may write
\[|A|^4=A^*AA^*A={A^*}^2A^2=A^4\]
and
\[|A^*|^4=AA^*AA^*=A^*A^2A^*={A^*}^4.\]
Since $A^2={A^*}^2$, we have $A^4={A^*}^4$, thereby $|A|^4=|A^*|^4$.
Upon passing to the unique positive square root, we obtain
$|A|^2=|A^*|^2$ or $A^*A=AA^*$. Since $A$ is already closed, the
normality of $A$ follows, as needed.
\end{proof}

\begin{rema}If $A\in B(H)$, then $A^2$ is self-adjoint if and only
if $A^2={A^*}^2$. This is not always the case when $A$ is closed and
densely defined. Indeed, in \cite{Dehimi-Mortad-CHERNOFF} we found a
closed densely defined operator $T$ such that
\[D(T^2)=D({T^*}^2)=\{0\}\]
(and so $T^2$ cannot be self-adjoint).

What about assuming that $D(T^2)$ is dense? This is still not
enough. Indeed, consider two self-adjoint operators $A$ and $B$ such
that $AB=BA$ on some common core but $A$ and $B$ do not commute
strongly. This is obviously not trivial and it is some kind of a
Nelson-like counterexample. To get the appropriate example, we
choose Schm\"{u}dgen's example (see Example 5.5 in
\cite{SCHMUDG-book-2012}). Then set
\[T=\left(
      \begin{array}{cc}
        0 & A \\
        B & 0 \\
      \end{array}
    \right)
\]
with $D(T)=D(B)\oplus D(A)$. Then $T$ is closed. Moreover,
$T^*=\left(
      \begin{array}{cc}
        0 & B \\
        A & 0 \\
      \end{array}
    \right)$ with $D(T^*)=D(A)\oplus D(B)$. Hence
\[T^2=\left(
      \begin{array}{cc}
        AB & 0 \\
        0 & BA\\
      \end{array}
    \right)=\left(
      \begin{array}{cc}
        BA & 0 \\
        0 & AB\\
      \end{array}
    \right)={T^*}^2\]
because $D(T^2)=D({T^*}^2)=D(AB)\oplus D(AB)$. Finally, $T^2$ cannot
be self-adjoint for if it were, then this would lead to the strong
commutativity of $A$ and $B$, which is impossible.
\end{rema}

The next result generalizes some known results in the bounded case.
Its proof is based upon a quite interesting paper by M. Uchiyama
(\cite{Uchiyama-1993-QUASINORMAL}) on quasinormality (both bounded
and unbounded operators). Unfortunately, many operator theorists
were unaware of it as could be guessed by looking at some papers
which appeared after Uchiyama's, where some of his results were
re-shown.

Before giving the alluded result, recall that when $T^n$ is densely
defined, then in general $T^{*n}\subsetneq (T^n)^*$ even when $T$
belongs to some nice class. For example, Jab{\l}o\'{n}ski et al.
constructed in \cite{Jablonski et al 2014} a quasinormal operator
$T$ such that $T^{*n}\subsetneq (T^n)^*$ for all $n\geq2$. So, the
next lemma seems to have some interest.

\begin{lem}\label{pooppopopooppopoopopiouiurthnnnn,l,;lc;x;c;xc:mcvlmcvlmc*}
Let $T$ be a quasinormal operator such that $T^n$ is densely defined
and $D[(T^n)^*]=D(T^n)$ (e.g. when $T^n$ is normal). Then
\[(T^n)^*=T^{*n}.\]
\end{lem}

\begin{proof}By the general theory, we already have that $T^{*n}\subset
(T^n)^*$. So, we only have to show that $D[(T^n)^*]\subset
D({T^*}^n)$. Since $T$ is quasinormal, $U|T|\subset |T|U$ where
$T=U|T|$ is the usual polar decomposition in terms of partial
isometries (see Proposition 1 in \cite{Stochel-Szafraniec-normal
extensions-II}, in fact $U|T|=|T|U$ as in
\cite{Uchiyama-1993-QUASINORMAL}). Hence $T^*=|T|U^*$ and
$U^*|T|\subset |T|U^*$. Since $T$ is quasinormal, $|T^n|=|T|^n$ (see
e.g. \cite{Jablonski et al 2014}). Moreover, $T^n$ is closed for $T$
is subnormal and closed (see \cite{Stochel-IEOT-2002}).

Therefore,
\[D(T^{*n})=D[(|T|U^*)^n]\supset D(U^{*n}|T|^n)=D(|T|^n)=D(|T^n|)=D(T^n)=D[(T^n)^*],\]
as wished.
\end{proof}

We are now in a position to state and prove a result on quasinormal
$n$th roots of normal operators (see
\cite{Pietrzycki-Stochel-follow-up-2020-Conjecture curto el al.} for
a related result).

\begin{thm}\label{quasinormal, Tn normal T is normal THM}
Let $T$ be a quasinormal (unbounded) operator such that $T^n$ is
normal for some $n\geq 2$. Then $T$ is normal.
\end{thm}

\begin{proof}By Corollary 3.1 in \cite{Uchiyama-1993-QUASINORMAL},
we know that
\[T^{*n}T^n=(T^*T)^n\geq (TT^*)^n\geq T^nT^{*n}\]
whenever $T$ is quasinormal and for all $n\geq2$.

A little digression, it was unclear to us which order relation was
used in the above inequalities given that M. Uchiyama did not
indicate the one he was using. However, in our case, it does not
affect our result given the assumptions we made in the theorem.
Indeed, all operators in the inequalities are self-adjoint (this
includes $T^nT^{*n}$ for by Lemma
\ref{pooppopopooppopoopopiouiurthnnnn,l,;lc;x;c;xc:mcvlmcvlmc*},
$(T^n)^*=T^{*n}$). A look at Page 230 in \cite{SCHMUDG-book-2012}
will shed light on this particular point.

Let us finish the proof now. By the normality of $T^n$, Lemma
\ref{pooppopopooppopoopopiouiurthnnnn,l,;lc;x;c;xc:mcvlmcvlmc*} and
the above inequalities, we may write
\[T^n(T^n)^*=(T^n)^*T^n=T^{*n}T^n=(T^*T)^n\geq (TT^*)^n\geq T^nT^{*n}=T^n(T^n)^*.\]
Therefore,
\[(T^*T)^n=(TT^*)^n~(=T^n(T^n)^*=(T^n)^*T^n).\]

Passing to the unique positive $nth$ root yields $|T|^2=|T^*|^2$,
that is, it yields the normality of $T$, as suggested.
\end{proof}

The next consequence might be known to some specialists, however, we
believe that it is not documented anywhere. Besides, the proof is
extremely easy once we fall back on a very recent result by P.
Pietrzycki and J. Stochel, whose proof is somewhat involved.

\begin{cor}\label{subnormal power normal is normal CORO}
Let $T$ be a closed subnormal (unbounded) operator such that $T^n$
is normal for some $n\geq 2$. Then $T$ is normal.
\end{cor}

\begin{proof}Since $T^n$ is normal, it is quasinormal. By Theorem
1.4 in \cite{Pietrzycki-Stochel-follow-up-2020-Conjecture curto el
al.}, we know that $T$ must be quasinormal. Theorem
\ref{quasinormal, Tn normal T is normal THM} then gives the
normality of $T$.

\end{proof}

As mentioned above, we present a generalization of the
Sebestyén-Tarcsay-Gesztesy-Schm\"{u}dgen reversed von Neumann
theorem. Remember in passing that when $TT^*$ and $T^*T$ are
self-adjoint, then $TT^*$ and $T^*T$ are closed and
$\sigma(TT^*),\sigma(T^*T)\subset \R$.

\begin{thm}\label{TT* T*T closed give T closed once it is closable
THM} Let $T$ be a densely defined closable such that
$\sigma(TT^*)\cup \sigma(T^*T)\neq \C$. Then $T$ is necessarily
closed.
\end{thm}

\begin{proof}Just write
\[A=\left(
      \begin{array}{cc}
        0 & T \\
        T^* & 0 \\
      \end{array}
    \right)
\]
which is closable. Then
\[A^2=\left(
        \begin{array}{cc}
          TT^* & 0 \\
          0 & T^*T \\
        \end{array}
      \right).
\]

By the assumption $\sigma(TT^*)\cup \sigma(T^*T)\neq \C$, we see
that neither $\sigma(TT^*)=\C$ nor $\sigma(T^*T)=\C$. Therefore,
$TT^*$ and $T^*T$ are both closed. Hence $A^2$ too is closed. Since
$\sigma(A^2)\neq \C$, Theorem \ref{THM main A2 closed} entails the
closedness of $A$ which, in turn, yields the closedness of $T$.
\end{proof}

Next, we give two simple characterizations of (unbounded) normal
operators.

\begin{cor}Let $T$ be a densely defined closable linear operator. Then
\[T\text{ is normal }\Longleftrightarrow \text{ $TT^*=T^*T$ and $\sigma(T^*T)\neq\C$.}\]
\end{cor}

\begin{proof}If $T$ is normal, then $TT^*=T^*T$ and $T$ is closed.
Hence $T^*T$ is self-adjoint and so $\C\neq \sigma(T^*T)\subset\R$.
Conversely, as $\sigma(T^*T)\neq\C$, $T^*T$ is closed. Hence $TT^*$
too is closed. Theorem \ref{TT* T*T closed give T closed once it is
closable THM} then yields the closedness of $T$ because
$\sigma(TT^*)\cup \sigma(T^*T)=\sigma(T^*T)\neq \C$.
\end{proof}

By calling on Proposition 1.1 in \cite{Stochel-IEOT-2002}, we have
the following result:

\begin{cor}Let $T$ be a densely defined closable linear operator such that $\sigma(TT^*)\cup \sigma(T^*T)\neq \C$.
Then either $T$ is normal or $TT^*\not\subset T^*T$ or
$T^*T\not\subset TT^*$.
\end{cor}

\begin{proof}As $\sigma(TT^*)\cup \sigma(T^*T)\neq \C$, we know that $T$ must be closed. Hence by Proposition 1.1 in \cite{Stochel-IEOT-2002},
either $TT^*=T^*T$ (i.e. $T$ is normal) or  $TT^*\not\subset T^*T$
or $T^*T\not\subset TT^*$.
\end{proof}

Before giving another related result, we give an example.

\begin{exa}(\cite{Mortad-div-lin-oper})
Consider the following two operators defined by
\[Af(x)=e^{2x}f(x) \text{ and } Bf(x)=(e^{-x}+1)f(x)\]
on their respective domains
\[D(A)=\{f\in L^2(\R):~e^{2x}f\in L^2(\R)\}\]
and \[D(B)=\{f\in L^2(\R):~e^{-2x}f, e^{-x}f\in L^2(\R)\}.\] Then
obviously $A$ is self-adjoint while $B$ closable without being
closed.

Now, the operator $BA$ defined by $BAf(x)=(e^{2x}+e^x)f(x)$ on
\[D(BA)=\{f\in L^2(\R):~e^{2x}f, e^xf\in L^2(\R)\}\]
is plainly self-adjoint. Readers may also check that $AB$ is not
self-adjoint.
\end{exa}

Inspired by this example, we have:

\begin{pro}It is impossible to find two unbounded linear operators $A$ and
$B$, one of them is closable (without being closed) and the other is
self-adjoint, such that both $BA$ and $AB$ are self-adjoint.
\end{pro}

\begin{proof}Let $B$ be a solely closable operator with domain $D(B)$, and
let $A$ be a self-adjoint operator with domain $D(A)$. Then set
\[T=\left(
      \begin{array}{cc}
        0 & B \\
        A & 0 \\
      \end{array}
    \right)
\]
which is defined on $D(T)=D(A)\oplus D(B)$. Clearly, $T$ is only
closable (i.e. $T$ is not closed). Now,
\[T^2=\left(
        \begin{array}{cc}
          BA & 0 \\
          0 & AB \\
        \end{array}
      \right).
\]

So if $BA$ and $AB$ were both self-adjoint, it would ensue that
$T^2$ is self-adjoint which, given the closability of $T$, would
yield the closedness of $T$ (by Corollary \ref{A2 s.a. A closable A
closed}). Since this is absurd, at least one of $AB$ and $BA$ must
be non-self-adjoint.
\end{proof}

There remains to investigate the very related question: Is there a
densely defined, closable and unclosed operator $A$ such that $A^2$
is closed, densely defined, and obeys $\sigma(A^2)=\C$? By looking
closely at the foregoing results, we actually see that the condition
$\sigma(A^2)=\C$ is superfluous.

A simple example is already available in the case of the absence of
the density of $D(A^2)$, as alluded to in the introduction. So,
things are more interesting when $D(A^2)$ is dense, and there is a
counterexample in this case too. But, to get the sought example, we
need a densely defined closed operator $T$ such that $D(T^2)=D(T)$.
Recall now that S. \^{O}ta (\cite{Ota-nilpotent-idempotent})
introduced the concept of an unbounded idempotent operator: Let $T$
be a non-necessarily bounded operator with a dense domain $D(T)$. We
say that $T$ is idempotent if $T^2$ is well defined and
\[T^2=T \text{ on }D(T).\]

Surprisingly, S. \^{O}ta did not provide any example of an unbounded
closed idempotent operator even though he did provide other examples
of unclosable idempotents. So, let us supply such an example:

\begin{exa}(\cite{Mortad-cex-BOOK})
Let $B$ be an unbounded closed operator with domain $D(B)\subset H$
and let $I$ be the identity operator on all of $H$. Define
\[T=\left(
      \begin{array}{cc}
        I & B \\
        0 & 0 \\
      \end{array}
    \right)
\]
with $D(T)=H\times D(B)$. Then $T$ is densely defined, closed and
unbounded. Since
\[D(T^2)=\{(x,y)\in H\times D(B):(x+By,0)\in H\times D(B)\}=D(T),\]
we see that
\[T^2=\left(
      \begin{array}{cc}
        I & B \\
        0 & 0 \\
      \end{array}
    \right)\left(
      \begin{array}{cc}
        I & B \\
        0 & 0 \\
      \end{array}
    \right)=\left(
      \begin{array}{cc}
        I &B \\
        0 & 0 \\
      \end{array}
    \right)=T.\]
\end{exa}

With this example at hand, we may now construct a densely defined,
closable and unclosed operator $A$ such that $A^2$ is densely
defined and closed (hence necessarily $\sigma(A^2)=\C$):

\begin{exa}\label{21/03/2021 EXAMPLE}(\cite{Mortad-cex-BOOK})
Let $T$ be a densely defined, unbounded and closed operator $T$ such
that $D(T^2)=D(T)\subset H$, and consider the identity operator on
$H$ restricted to $D(T)$, noted $I_{D(T)}$. Next, let
\[A=\left(
      \begin{array}{cc}
        0 & T \\
        I_{D(T)} & 0 \\
      \end{array}
    \right)
\]
where $D(A)=D(T)\times D(T)$. Clearly, $A$ is closable but unclosed.
Let $0_{D(T)}$ and $T_{D(T^2)}$ designate the restrictions of the
zero operator and of $T$ to the subspaces $D(T)$ and $D(T^2)$
respectively. We then have
\[A^2=\left(
      \begin{array}{cc}
        0 & T \\
        I_{D(T)} & 0 \\
      \end{array}
    \right)\left(
      \begin{array}{cc}
        0 & T \\
        I_{D(T)} & 0 \\
      \end{array}
    \right)=\left(
      \begin{array}{cc}
        T & 0_{D(T)} \\
        0_{D(T)} & T_{D(T^2)} \\
      \end{array}
    \right)=\left(
      \begin{array}{cc}
        T & 0 \\
        0 & T \\
      \end{array}
    \right).\]
Thus, $A^2$ is patently closed on the dense $D(T)\times D(T)$, as
wished.
\end{exa}

\begin{cor}If $T=\left(
      \begin{array}{cc}
        I & B \\
        0 & 0 \\
      \end{array}
    \right)$, then $\sigma(T)=\C$ for any unbounded closed (even self-adjoint) operator $B$, where
$I\in B(H)$ is the usual identity.
\end{cor}

\begin{proof}Let $A$ be as in Example \ref{21/03/2021 EXAMPLE}, and
so $A^2=\left(
      \begin{array}{cc}
        T & 0 \\
        0 & T \\
      \end{array}
    \right)$. Hence $\sigma(A^2)=\sigma(T)$. But, we already
    observed above that $\sigma(A^2)=\C$, i.e. $\sigma(T)=\C$, as
    needed.
\end{proof}

We would like to generalize Corollary \ref{17/05/2021 CORO} to the
case $A^n$ where $n\geq3$? A direct generalization is obviously
false. For instance, if $A=e^{{2\pi i}/3}I$, then $A$ is unitary and
non-self-adjoint, yet $A^3=I$ is obviously positive. If $A^n$ and
$A^{n+1}$ are both positive and self-adjoint for some $n$, then $A$
is self-adjoint and positive whenever it is hyponormal (non
necessarily bounded). After having shown this result, a better
version came to our minds. It is interesting to see how elementary
number theory may be called on to show results in (unbounded)
operator theory.

\begin{thm}\label{hyponormal powers co-primes s.a. THM}
Let $A$ be a non-necessarily bounded hyponormal operator with domain
$D(A)\subset H$. If $A^p$ and $A^q$ are two self-adjoint operators,
where $p$ and $q$ are two co-prime numbers, then $A$ is
self-adjoint.
\end{thm}

\begin{proof}
By a similar method as in the proof of Theorem \ref{THM main A2
closed} or a consequence of Theorem
\ref{124589763254789995412300000000000000333265 THM} below, $A$ is
closed. Let $\lambda$ be in $\sigma(A)$. Hence $\lambda^p\in\R$ and
$\lambda^q\in\R$. If $\lambda=0$, then $\sigma(A)=\{0\}$, whereby
$A$ becomes self-adjoint.

Now, let $\lambda\neq0$. By Bézout's theorem in arithmetic, we
know that $ap+bq=1$ for some integers $a$ and $b$.  Therefore,
$\lambda ^{ap}, \lambda^{bq}\in \R-\{0\}$, and so
$\lambda^{ap+bq}\in \R-\{0\}$. In other words, $\lambda \in
\R-\{0\}$. So in all cases, $\lambda$ must be real. This says that
$A$ has a real spectrum, and whence $A$ becomes self-adjoint.
\end{proof}

\begin{rema}
A similar result holds by replacing "self-adjoint" by "positive" in
both the assumption and the conclusion, in the case of bounded
operators; and by replacing "self-adjoint" by "self-adjoint and
positive" in the unbounded case.
\end{rema}

\begin{rema}
The assumption $p$ and $q$ being relatively prime numbers is
essential. To show its importance, consider as above $A=e^{{2\pi
i}/3}I$. Then both $A^3$ and $A^6$ are positive, yet $A$ is not even
self-adjoint.
\end{rema}

The idea of the proof of the previous theorem may be applied to show
some similar results. As is known, square roots of normal operators
need not be normal, even when $\dim H=2$. Requiring a square root of
a normal operator to be invertible still does not help (cf. Theorem
3 in \cite{kitt-normality-primes-1984}). The next result may
therefore be useful.

\begin{thm}\label{invertible two powers co-primes normal THM}
Let $A$ be a boundedly invertible non-necessarily bounded operator
with domain $D(A)$. If $A^p$ and $A^q$ are normal, where $p$ and $q$
are two co-prime numbers, then $A$ is normal.
\end{thm}

\begin{proof}As above, we have by Bézout's theorem that $ap+bq=1$ for some integers $a$ and
$b$. Necessarily, one of $ap$ and $bq$ has to be negative, and WLOG
assume it is $bq$. Since $A$ is boundedly invertible, we have
$A^{-1}A\subset I$ and $AA^{-1}=I$. Hence
\[A^{bq}A^{ap}\subset A=A^{ap}A^{bq}.\]

Since $A$ is boundedly invertible, $A^{bq}\in B(H)$. Since $A^p$ and
$A^q$ are also normal, $A^{ap}$ and $A^{bq}$ remain normal. In other
words, the bounded and normal $A^{bq}$ commutes with the normal
$A^{ap}$. Hence, by say Theorem 2.2 in \cite{Meziane-Mortad-I},
\[A^{ap}A^{bq}=A\]
is normal, as needed.
\end{proof}

\textit{Mutatis mutandis}, a similar result holds for self-adjoint
as well as self-adjoint positive operators. Consulting Lemma 3.1 in
\cite{Jung-Mortad-Stochel} comes in handy, and so we omit the proof.

\begin{thm}
Let $A$ be a boundedly invertible non-necessarily bounded operator
with domain $D(A)$. If $A^p$ and $A^q$ are self-adjoint (resp.
self-adjoint and positive), where $p$ and $q$ are two co-prime
numbers, then $A$ is self-adjoint (resp. self-adjoint and positive).
\end{thm}

\begin{rema}
To see why the condition "$p$ and $q$ being co-prime numbers" may
not just dropped, it suffices to take an invertible non-normal
square root $A$ of the identity matrix, and then $A^2=A^4=I$. An
explicit example would be $A=\left(
        \begin{array}{cc}
          2 & 1 \\
          -3 & -2 \\
        \end{array}
      \right)$.
\end{rema}

It is well known that if $S$ is a positive self-adjoint operator
which commutes with some $R\in B(H)$, i.e. $RS\subset SR$, then
$R\sqrt{S}\subset \sqrt S R$ where $\sqrt S$ designates the unique
positive self-adjoint square root of $S$. See e.g.
\cite{Sebestyen-Tarcsay-self-adjoint square ROOT} for a new proof.

What about arbitrary roots? The answer is negative even on
finite-dimensional spaces. For example, consider a $2\times 2$
matrix, noted $T$. Then $T$ commutes with $A^2=I$, and we can
obviously choose $T$ such that it does not commute with an
\textit{invertible} square root of $I$ (for example the matrix $A$
in the preceding remark). The same $T$ also commute with $A^4$. So,
a judicious choice of exponents must be made if we want a positive
result.

\begin{pro}
Let $A$ be a non-necessarily bounded operator which is boundedly
invertible, and let $T\in B(H)$. If $T$ commutes with both $A^p$ and
$A^q$, i.e. $TA^p\subset A^pT$ and $TA^q\subset A^qT$ for some
relatively prime numbers $p$ and $q$, then $TA\subset AT$, i.e. $T$
commutes with $A$.
\end{pro}

\begin{proof}
The proof \textit{iterum} uses Bézout's theorem. So, $ap+bq=1$
for some integers $a$ and $b$ (choose $bq$ to be negative). Since
$TA^p\subset A^pT$ and $TA^q\subset A^qT$, it follows that
$TA^{ap}\subset A^{ap}T$ and $TA^{bq}=A^{bq}T$ for $A^{bq},T\in
B(H)$. Hence
\[TA^{ap}\subset A^{ap}T\Longrightarrow TA^{ap}A^{bq}\subset A^{ap}TA^{bq}=A^{ap}A^{bq}T.\]
As in the proof of Theorem \ref{invertible two powers co-primes
normal THM}, we derive $A^{ap}A^{bq}=A$. Therefore, $TA\subset AT$,
as wished.
\end{proof}

\section{Main Results: The case of polynomials}

Now, we treat the more general case of $p(A)$ where $p$ is a complex
polynomial.

\begin{thm}\label{124589763254789995412300000000000000333265 THM} Let $p$ be a complex polynomial of one variable of degree
$n$. Assume that $A$ is a closable operator in a Hilbert (or Banach)
space such that ($p(A)$ is densely defined and)
$\sigma[p(A)]\neq\C$. Then $A$ is closed.
\end{thm}

\begin{rema} Observe that the closedness of $p(A)$ is tacitly assumed
  because $\sigma[p(A)]\neq\C$.
\end{rema}

\begin{proof}Let $\lambda$ be in $\C\setminus \sigma[p(A)]$. We may also assume that the leading coefficient of $p(A)$ equals 1. By the fundamental theorem of Algebra, we know that
there are complex numbers $\mu_1$, $\mu_2$, $\cdots$, $\mu_n$ such
that
\[p(z)-\lambda=(z-\mu_1)(z-\mu_2)\cdots (z-\mu_n)\]
where $z\in\C$. Hence
\[p(A)-\lambda I=(A-\mu_1 I)(A-\mu_2I)\cdots (A-\mu_nI).\]
The previous is in effect a full equality for \[D[p(A)-\lambda
I]=D[(A-\mu_1 I)(A-\mu_2I)\cdots (A-\mu_nI)]=D(A^n)\] (which may be
checked using a proof by induction). Since $p(A)-\lambda I$ is
boundedly invertible, we have that
\[(A-\mu_1 I)(A-\mu_2I)\cdots
(A-\mu_nI)B=I\] for some $B\in B(H)$. But $(A-\mu_2I)\cdots
(A-\mu_nI)B\in B(H)$. Indeed, since $A$ is closable, so is
$A-\mu_nI$ and so $(A-\mu_nI)B$ is closable. Hence $(A-\mu_nI)B\in
B(H)$ for $D[(A-\mu_nI)B]=H$.

Now, $(A-\mu_{n-1})(A-\mu_nI)B\in B(H)$ by using a similar argument.
By induction, it may then be shown that $(A-\mu_2I)\cdots
(A-\mu_nI)B\in B(H)$. Thus, $A-\mu_1 I$ is right invertible. Since
\[(A-\mu_1 I)(A-\mu_2I)\cdots (A-\mu_nI)=(A-\mu_2 I)(A-\mu_3I)\cdots (A-\mu_n I)(A-\mu_1I),\]
it is seen that $A-\mu_1 I$ is one-to-one. Finally, as in the proof
of Theorem \ref{THM main A2 closed}, we may conclude that $A$ is
closed, as suggested.
\end{proof}

As in the case of $p(z)=z^2$, we have:

\begin{cor}Let $p$ be a complex polynomial of one variable. Assume that $A$ is a closable non-closed operator in a Hilbert (or
Banach) space such that $p(A)$ is closed. Then
\[\sigma[p(A)]=\C.\]
\end{cor}

Next, we generalize Proposition \ref{25/0/2020}.

\begin{thm}\label{26/07/2020 THM}
Let $p$ be a polynomial of one variable of degree $n$. Assume that
$A$ is a symmetric (not necessarily densely defined) operator in a
Hilbert space $H$ such that $p(A)$ is self-adjoint.  Then $A$ is
self-adjoint.
\end{thm}

\begin{proof}
Since $A$ is symmetric, to show it is self-adjoint we may instead
show that $\ran(A-\mu I)=H$ and $\overline{\ran(A-\overline{\mu}
I)}=H$ where $\overline{\mu}$ is the complex conjugate of $\mu$ (see
Proposition 3.11 in \cite{SCHMUDG-book-2012}).

Let $\lambda\in\C\setminus \sigma[p(A)]$. As above, we may assume
that the leading coefficient of $p(A)$ is equal to 1. Write
\[p(A)-\lambda I=(A-\mu_1 I)(A-\mu_2I)\cdots (A-\mu_nI)\]
where $\mu_i,i=1,\cdots,n$ are complex numbers. As above, it can be
shown that $A-\mu_1 I$ is right invertible, i.e. it is surjective.
Since $p(A)-\lambda I$ is boundedly invertible, so is $[p(A)-\lambda
I]^*$. But
\[[p(A)-\lambda I]^*=p(A)-\overline{\lambda} I\]
as $p(A)$ is self-adjoint. Since $A$ is symmetric and using some
standard properties, we see that
\begin{align*}(A-\overline{\mu_1} I)(A-\overline{\mu_2}I)\cdots (A-\overline{\mu_n}I)&\subset (A^*-\overline{\mu_1}
I)(A^*-\overline{\mu_2}I)\cdots (A^*-\overline{\mu_n}I)\\&\subset
[(A-\mu_1 I)(A-\mu_2I)\cdots
(A-\mu_nI)]^*\\&=p(A)-\overline{\lambda} I.
\end{align*}
Therefore,
\[(A-\overline{\mu_1} I)(A-\overline{\mu_2}I)\cdots (A-\overline{\mu_n}I)=p(A)-\overline{\lambda} I\]
as both sides have the same domain, namely $D(A^n)$. Thus,
$A-\overline{\mu_1}I$ too is surjective. Accordingly, $A$ is
self-adjoint.
\end{proof}

\begin{cor}
Let $p$ be a polynomial of one variable of degree $n$ with real
coefficients. Assume that $A$ is a symmetric (not necessarily
densely defined) operator in a Hilbert space $H$ such that $p(A)$ is
normal. Then $A$ is self-adjoint.
\end{cor}

\begin{proof}
Since $A$ is symmetric, so is $p(A)$ because $p$ has real
coefficients. Hence
\[p(A)\subset p(A^*)\subset [p(A)]^*,\]
i.e. $p(A)$ is symmetric. Given its normality, it becomes
self-adjoint and so $A$ is self-adjoint by Theorem \ref{26/07/2020
THM}.
\end{proof}

Amazingly, M. Uchiyama showed that a symmetric quasinormal operator
is self-adjoint (the last corollary in
\cite{Uchiyama-1993-QUASINORMAL}). So, the next consequence is
stated without proof.

\begin{cor}Let $p$ be a polynomial of one variable of degree $n$ with real
coefficients. Assume that $A$ is a symmetric (not necessarily
densely defined) operator in a Hilbert space $H$ such that $p(A)$ is
quasinormal. Then $A$ is self-adjoint.
\end{cor}

The results above may not be generalized to functions $f(A)$ even
when the latter is well defined. Let us give a counterexample.

\begin{exa}
Let $A$ be a densely defined symmetric (hence closable) unclosed
operator such that $A^*A$ is self-adjoint. Then $|A|^2$ is
self-adjoint while $A$ is not closed. An explicit realization
(inspired by one which appeared in \cite{sebestyen-tarcsay-TT* has
an extension}) reads:

Let $T=i\frac{d}{dx}$ be defined on $H^1(\R)=\{f\in L^2(\R):f'\in
L^2(\R)\}$. Then $T$ is self-adjoint (hence $T^*T=T^2$ is
self-adjoint). Set $A=T|_{H^2(\R)}$ where $H^2(\R)=\{f\in
L^2(\R):f''\in L^2(\R)\}$. Then $A$ is not closed and $T^*=A^*$.
Since
\[D(A^*A)=D(T^2)=H^2(\R),\]
it follows that $A^*A=T^2$. Thus, since $T^2$ is self-adjoint so is
$A^*A$.
\end{exa}

\section*{Conjectures}
\begin{enumerate}
  \item If it can be shown that a closed hyponormal operator having
  a quasinormal power is necessarily quasinormal, then a similar
  idea as in Corollary \ref{subnormal power normal is normal CORO}
  may be applied to show that a closed hyponormal operator having a
  normal power is automatically normal. This result is perhaps known
  to a few specialists, but we could not find it anywhere. In addition,
  if a proof exists somewhere, then the one we have proposed here would be simpler.
  \item Let $A$ be a closed densely defined paranormal operator, with $\sigma(A)\subset \R$. Must $A$ be self-adjoint? If this is true, then Theorem \ref{hyponormal powers co-primes s.a.
  THM} could be improved by replacing the hyponormality assumption
  by a paranormality one.
  \item Powers of (unbounded) paranormal operators are normal (as in say \cite{Stochel-Sza-subnormality-BOOK-unpublished?}), and so are inverses of paranormal operators (as in \cite{Bala-Ramesh-paranormal} or in \cite{Stochel-Sza-subnormality-BOOK-unpublished?}).
  So can Theorem \ref{invertible two powers co-primes normal THM} be
  generalized to closed paranormal operators?
\end{enumerate}

\end{document}